\documentclass[a4paper, 10pt, reqno, english]{amsart}

\usepackage{amsmath, amssymb, amsthm, amsfonts, graphicx}
\usepackage{graphicx,subfigure}
\usepackage[pdfencoding=auto]{hyperref}
\usepackage{dsfont}
\usepackage{xcolor}
\usepackage[utf8]{inputenc}
\usepackage{microtype}
\usepackage[nobysame, alphabetic]{amsrefs}
\usepackage[capitalize]{cleveref}

\newtheorem{theorem}{Theorem}[section]
\newtheorem{lemma}[theorem]{Lemma}

\newtheorem{corollary}[theorem]{Corollary}
\theoremstyle{definition}
\newtheorem{question}[theorem]{Question}
\newtheorem{definition}[theorem]{Definition}

\newcommand{\R}{\mathds{R}}
\newcommand{\rr}{\mathds{R}}

\newcommand{\zz}{\mathds{Z}}

\DeclareMathOperator{\conv}{conv}

\DeclareMathOperator{\ind}{Ind}

\makeatletter
\def\blfootnote{\gdef\@thefnmark{}\@footnotetext}
\makeatother

\title{Lifting methods in mass partition problems}

\date{}
\author{Pablo Sober\'{o}n}
\address{Baruch College, City University of New York.  New York, NY 10010}
\email{pablo.soberon-bravo@baruch.cuny.edu}
\author{Yuki Takahashi}
\address{Grinnell College. Grinnell, IA 50112}
\email{takahash@grinnell.edu}

\thanks{This project was done as part of the 2021 New York Discrete Math REU, funded by NSF grant DMS 2051026.  Sober\'on’s research is supported by NSF grant DMS 2054419 and a PSC-CUNY TRADB52 award.   Takahashi's research is supported by Grinnell College Internship Funding.}

\begin{document}

\begin{abstract}
    Many results in mass partitions are proved by lifting $\rr^d$ to a higher-dimensional space and dividing the higher-dimensional space into pieces.  We extend such methods to use lifting arguments to polyhedral surfaces.  Among other results, we prove the existence of equipartitions of $d+1$ measures in $\rr^d$ by parallel hyperplanes and of $d+2$ measures in $\rr^d$ by concentric spheres.
    
    For measures whose supports are sufficiently well separated, we prove results where one can cut a fixed (possibly different) fraction of each measure either by parallel hyperplanes, concentric spheres, convex polyhedral surfaces of few facets, or convex polytopes with few vertices.
   
\end{abstract}

\maketitle

\section{Introduction}

In a standard mass partition problem, we are given measures or finite families of points in a Euclidean space and we seek to partition the ambient space into regions that meet certain conditions.  Some conditions determine how we split the measures and the sets of points.  For instance, in an equipartition, we ask that each part has the same size in each measure or contains the same number of points of each set.  Some conditions restrict the types of partitions which are allowed, such as partition by a single hyperplane.  Determining whether such partitions always exist leads to a rich family of problems.  Solutions to these problems often require topological methods and can have computational applications \cites{matousek2003using, Zivaljevic2017, RoldanPensado2021}.  The quintessential mass partition result is the ham sandwich theorem, conjectured by Steinhaus and proved by Banach \cite{Steinhaus1938}.

\begin{theorem}[Ham sandwich theorem]
	Let $d$ be a positive integer and $\mu_1, \ldots, \mu_d$ be finite measures of $\rr^d$.  Then, there exists a hyperplane $H$ of $\rr^d$ so that its two closed half-spaces $H^+$ and $H^-$ satisfy
	\begin{align*}
		\mu_i (H^+) & \ge \frac{1}{2}\mu_i (\rr^d), \\
		\mu_i (H^-) & \ge \frac{1}{2}\mu_i (\rr^d) \qquad \mbox{for $i=1,\ldots,d$.}
	\end{align*}
\end{theorem}

If we further ask that $\mu_i(H') = 0$ for each hyperplane $H'$ and every $i=1,\ldots, d$, the inequalities above are equalities.  Stone and Tukey proved the ham sandwich theorem for general measures \cite{Stone:1942hu}.  They also proved the polynomial ham sandwich theorem which states that \textit{any $\binom{d+k}{k}-1$ measures in $\rr^d$ can be halved with a polynomial on $d$ variables of degree at most $k$.}  Even though this is a far-reaching generalization of the ham sandwich theorem, its proof relies on a simple trick.  We lift $\rr^d$ to $\rr^{\binom{d+k}{k}-1}$ by the Veronese map and apply the ham sandwich theorem in the higher-dimensional space.

In this paper, we prove several mass partition results by lifting $\rr^d$ to higher-dimensional spaces, particularly $\rr^{d+1}$, in new ways.  In \cref{sec:sines-and-spheres}, we revisit a known result about equipartitions of measures with spheres and prove a new result about equipartitions of three measures in $\rr^2$ using a sinusoidal curve of fixed period.  Then, instead of lifting to higher-dimensional spaces via smooth maps, such as the Veronese maps, we lift to polyhedral surfaces in $\rr^{d+1}$.  This forces us to use the ham sandwich theorem for general measures, which is interesting on its own.

One of our main results is the following ham sandwich theorem for parallel hyperplanes.

\begin{theorem}\label{thm:parallel-hyperplanes}
Let $d$ be a positive integer and $\mu_1, \ldots, \mu_{d+1}$ be $d+1$ finite measures in $\rr^d$, each absolutely continuous with respect to the Lebesgue measure.  Then, there exist two parallel hyperplanes $H_1, H_2$ so that the region between them contains exactly half of each measure.
\end{theorem}

If there is a hyperplane $H_1$ that halves all measures, we can consider $H_2$ to be at infinity.  If $H_1$ and $H_2$ are not required to be parallel, we present with \cref{thm:simple-wedge} a simple proof that \textit{for any $d+1$ finite measures in $\rr^d$ there exist two closed half-spaces whose intersection contains exactly half of each measure}.  The intersection of two half-spaces is called a wedge.  The fact that $d+1$ measures in $\rr^d$ can be halved by a wedge was first proved by B\'ar\'any and Matou\v{s}ek in dimension two \cite{Barany:2001fs} and later generalized to $\rr^d$ by Schnider \cite{Schnider:2019ua}.  For $\rr^2$, Bereg presented algorithmic approaches for the discrete version which show that more conditions can be imposed to the wedge \cite{Bereg:2005voa}.

The proof requires a new Borsuk--Ulam type theorem about direct products of spheres and Stiefel manifolds, \cref{thm:new-topological-result}, which we describe in \cref{sec:polyhedral-lift}.  As a corollary, we combine \cref{thm:parallel-hyperplanes} with known lifting techniques.  We nickname the following result as the ``bagel ham sandwich theorem'', due to how it looks in $\rr^2$.

\begin{corollary}[Bagel ham sandwich theorem]\label{cor:Bagel}
Let $d$ be a positive integer and $\mu_1, \ldots, \mu_{d+2}$ be $d+2$ finite measures in $\rr^d$, each absolutely continuous with respect to the Lebesgue measure.  Then, we can find two concentric spheres in $\rr^d$ so that the closed region between them has exactly half of each measure.
\end{corollary}

\cref{thm:parallel-hyperplanes} is optimal, as the region between the two hyperplanes is convex.  One can simply take $d+1$ measures concentrated each around a vertex of a simplex and a final measure concentrated around the barycenter of the simplex to show that the result is impossible with $d+2$ measures.  The problem of cutting the same fraction for a family of measures with a single convex set has been studied before \cites{Akopyan:2013jt, Blagojevic:2007ij}, which we revisit in \cref{sec:remarks}.  \cref{thm:parallel-hyperplanes} is also related to the problem of halving measures in $\rr^d$ using hyperplane arrangements.  Langerman conjectured that \textit{any $dn$ measures in $\rr^d$ can be simultaneously halved by a chessboard coloring induced by $n$ hyperplanes} \cites{Barba:2019to, Hubard2020}.  For $n=2$, this has been confirmed for $2d-O(\log d)$ measures \cite{Blagojevic2018}.  If the hyperplanes are required to be parallel, this reduces the dimension of the space of possible partitions from $2d$ to $d+1$, matching the number of measures in \cref{thm:parallel-hyperplanes}.

General mass partition results like the ham sandwich theorem can halve many measures simultaneously.  If we want to cut a fixed (but possibly different) fraction of each measure, conditions need to be imposed.  For example, if two measures coincide, it is impossible to find a half-space that contains exactly half of one and one third of the other.

The first result with arbitrary sizes for each measure was proved by Hugo Steinhaus in dimensions two and three \cite{Steinhaus1945}.  He required the support of the measures to be well separated, meaning that the supports of any set of measures could be separated from the supports of the rest by a hyperplane.  This condition was sufficient to guarantee the existence of a half-space cutting a fixed fraction of several measures.  This result was rediscovered and extended to high dimensions independently by B\'ar\'any, Hubard, and Jer\'onimo and by Breuer \cites{Barany:2008vv, Breuer2010}.

\begin{theorem}[B\'ar\'any, Hubard, Jer\'onimo 2008; Breuer 2010]\label{thm:BHJ}
Let $d$ be a positive integer and $\mu_1, \ldots, \mu_d$ be $d$ finite measures in $\rr^d$, each absolutely continuous with respect to the Lebesgue measure so that their supports $K_1, \ldots, K_d$ are well separated.  Let $\alpha_1, \ldots, \alpha_d$ be real numbers in $(0,1)$.  Then, there exists a half-space $H$ so that
\[
\mu_i(H) = \alpha_i \cdot \mu_i (\rr^d) \qquad \mbox{for }i=1,\ldots, d.
\]
\end{theorem}

The proof of B\'ar\'any, Hubard, and Jer\'onimo uses Brouwer's fixed point theorem.  Breuer's proof uses the Poincar\'e--Miranda theorem, which is equivalent to Brouwer's fixed point theorem but has a significantly different formulation.  Steinhaus' proof is quite different and uses the Jordan curve theorem.  In \cref{sec:well separated} we present a new proof of \cref{thm:BHJ} that uses a degree argument.  We extend \cref{thm:parallel-hyperplanes} in a similar way for well separated measures.

\begin{theorem}\label{thm:parallel-hyperplanes-separated}
Let $d$ be a positive integer and $\mu_1, \ldots, \mu_{d+1}$ be $d+1$ finite measures in $\rr^d$, each absolutely continuous with respect to the Lebesgue measure.  Suppose that the supports $K_1, \ldots, K_{d+1}$ of $\mu_1, \ldots, \mu_{d+1}$ are well separated.  Let $\alpha_1, \ldots, \alpha_{d+1}$ be real numbers in $(0,1)$. Then, there exist two parallel hyperplanes $H_1, H_2$ in $\rr^d$ so that the region $A$ between them satisfies
\[
\mu_i(A) = \alpha_i \cdot \mu_i(\rr^d) \qquad \mbox{for all }i=1,\ldots, d+1.
\]
\end{theorem}

We also combine the results with partitions with few hyperplanes and those of partitions using a single convex.  We exhibit conditions for measures in $\rr^d$ that guarantee the existence of a (possibly unbounded) convex polyhedron of few facets which contains a fixed fraction of each measure or the existence of a convex polytope with few vertices that contains a fixed fraction of each measure.  This is done in \cref{sec:polyhedral}.  These results work with an arbitrary number of measures in $\rr^d$.

Finally, we revisit a mass partition result by Akopyan and Karasev that uses a lifting argument in its proof.  Akopyan and Karasev proved that for any positive integer $n$ and any $d+1$ measures in $\rr^d$, there exists a convex set $K$ whose measure is exactly $1/n$ of each measure.  We extend the methods from \cref{sec:polyhedral-lift} to bound the complexity of $K$ by writing it as the intersection of few half-spaces.

\begin{theorem}\label{thm:same-fraction}
Let $n,d $ be positive integers and $\mu_1, \ldots, \mu_{d+1}$ be $d+1$ finite measures in $\rr^d$, each absolutely continuous with respect to the Lebesgue measure.  There exists a convex set $K$, such that $K$ is the intersection of $\sum_{j=1}^r k_j(p_j-1)p_j$ half-spaces and
\[
\mu_i(K) = \frac{1}{n}\mu_i(\rr^d) \qquad \mbox{for all }i=1,\ldots, d+1,
\]
where $n=p_1^{k_1}\ldots p_r^{k_r}$ is the prime factorization of $n$.
\end{theorem}

We conclude in \cref{sec:remarks} with remarks and open problems.

\section{Equipartition with spheres and sine curves}\label{sec:sines-and-spheres}
While the traditional ham sandwich theorem simultaneously halves $d$ measures in $\rr^d$ by a hyperplane, we can simultaneously halve $d+1$ or more measures in $\rr^d$ if we increase the complexity of the cut.  The following theorem is a consequence of Stone and Tukey's polynomial ham sandwich theorem \cite{Stone:1942hu} and was one of Stone and Tukey's first examples of their main results.  It was also proved in dimension two by Hugo Steinhaus in 1945 \cite{Steinhaus1945} using a particular parametrizations of the space of circles in $\rr^2$.  We present a new proof with a stereographic projection.

\begin{theorem}\label{thm:circular-sandwich}
Let $d$ be a positive integer and let $\mu_1, \ldots, \mu_{d+1}$ be $d+1$ finite measures in $\R^d$, each absolutely continuous with respect to the Lebesgue measure.  Then, there exists either a sphere or a hyperplane that simultaneously splits each measure by half.
\end{theorem}

\begin{proof}
We first embed $\rr^d$ to $\rr^{d+1}$ by appending a coorinate $1$ to each point, so $x \mapsto (x,1) \in \rr^{d+1}$.  Then, we apply $r: \rr^{d+1}\setminus\{0\} \to \rr^{d+1}\setminus\{0\}$ the inversion centered at $0$ with radius $1$.  This is a transformation that sends spheres containing the origin to hyperplanes and hyperplanes to spheres containing the origin.  Hyperplanes containing the origin are fixed set-wise by the inversion and we consider them as degenerate spheres as well.  Restricted to  the embedding of $\rr^d$, the inversion is a stereographic projection to the sphere $S$ of radius $1/2$ centered at $(0,\ldots,0,1/2)$ by rays through the origin.  We also know that $r \circ r$ is the identity.

When we lift the measures $\mu_1, \ldots, \mu_{d+1}$ to $\rr^{d+1}$ and apply $r$, we get measures $\sigma_1, \ldots, \sigma_{d+1}$ on $S$.  By the ham sandwich theorem in $\rr^{d+1}$, there exists a hyperplane $H$ halving each of $\sigma_1, \ldots, \sigma_{d+1}$.  Since $r(H)$ is a sphere in $\rr^{d+1}$, it intersects the embedding of $\rr^d$ in a $(d-1)$-dimensional sphere halving each of $\mu_1, \ldots, \mu_{d+1}$, as we wanted.  The only exceptional case is if $H$ contains the origin, in which case $r(H) = H$, which gives us an equipartition by a hyperplane in $\rr^d$.
\end{proof}

Similarly, the idea of lifting allows for an visual and intuitive proof of the existence of equipartitions of three sets in $\R^2$ by a sine wave.  By a sine wave of period $\alpha$ we mean the graph of a function of the form $y = r + A \sin (2\pi x / \alpha + s)$, for real numbers $A,r,s$.

\begin{theorem}\label{thm:cylinder}
Let $\alpha > 0$ be a real number.  Given three finite measures $\mu_1, \mu_2, \mu_3$ in $\R^2$, each absolutely continuous with respect to the Lebesgue measure, there exists a sine wave of period $\alpha$ halving each measure.
\end{theorem}

  We allow ``degenerate'' sine waves of period $\alpha$.  A degenerate sine wave of period $\alpha$ is formed by taking two vertical lines intersecting the $x$-axis in the interval $[0,\alpha)$ and making a translated copy in each interval of the form $n\alpha + [0,\alpha)$.  This set induces a chessboard coloring of the plane into two regions.  We can think of this as the limit of a sequence of sine waves of period $\alpha$ of increasing ampitude.

\begin{proof}
    We prove the result for $\alpha = 2\pi$, as the two cases are equivalent after a scaling argument.  We wrap $\rr^d$ around the cylinder $C$ in $\rr^3$ with equation $x^2 + z^2 = 1$ with the function.
    
    \begin{align*}
        f: \rr^2 & \to C \\
        (x,y) & \mapsto \left(\cos\left(x\right), y, \sin\left(x\right)\right)
    \end{align*}

    Let $\sigma_1, \sigma_2, \sigma_3$ be the measures that $\mu_1, \mu_2, \mu_3$ induce on $C$ by this lifting, respectively.  We apply the ham sandwich theorem to these three measures in $\rr^3$.  Therefore, we can find a plane $H=\{(x,y,z): ax + by + cz = d\}$ that halves each of $\sigma_1, \sigma_2, \sigma_3$.  When we pull $H \cap C$ back to $\rr^2$, we get the set of points $(x,y)$ that satisfy $a \cos(x) + b y + c\sin(x) = d$.  Since a linear combination of the sine and cosine functions is a sinusoid with the same period but possibly different amplitude and phase shift, we have $a \sin(x) + c \cos (x) = A\sin(x + s)$ for some $A$ and $s$.
    
    \begin{figure}[ht!]
    \centering
    \includegraphics[width=.7\textwidth]{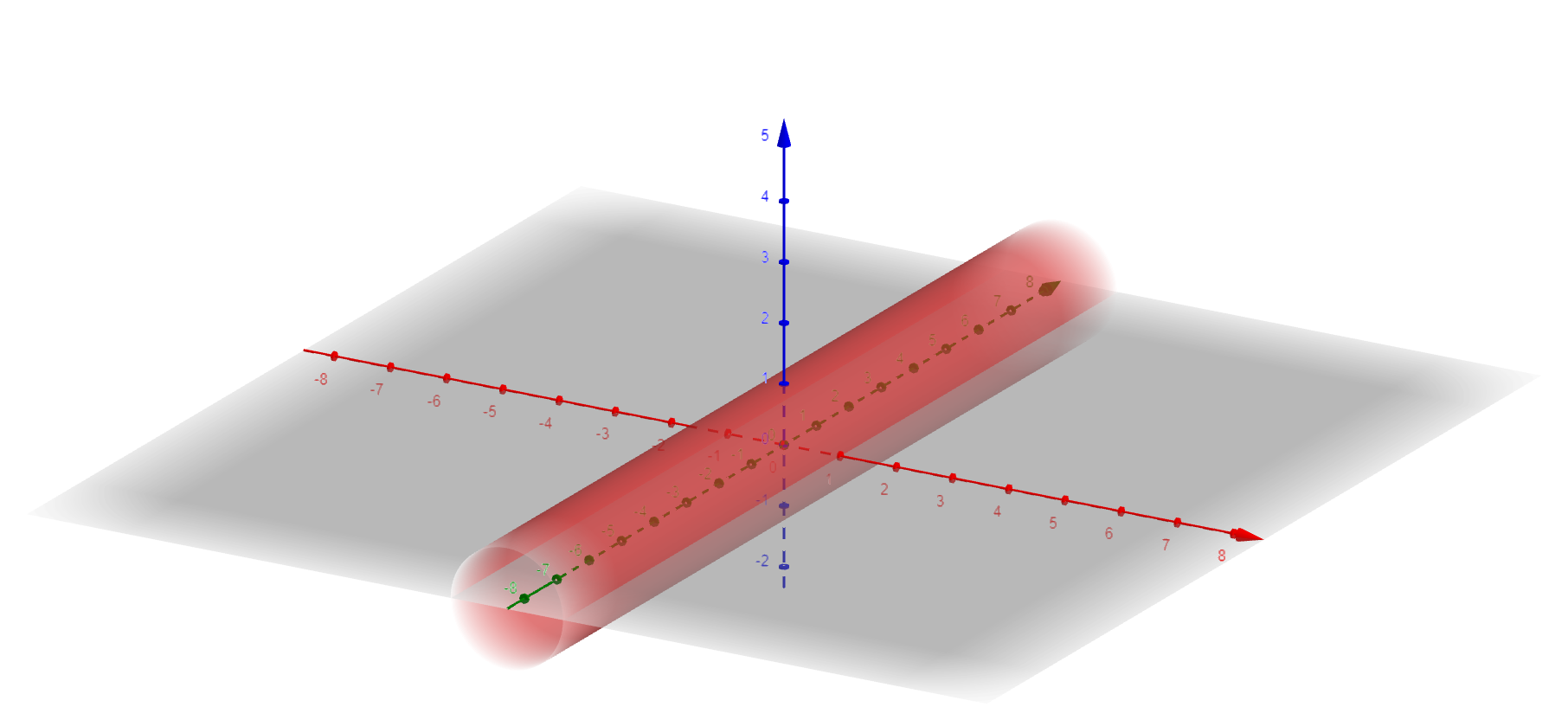}
    \caption{An example of a cylinder $C$ with period $\alpha=2\pi$.  The lift $\rr^2 \to C$ is not an injective function but this does not cause a problem.}
    \label{fig:cylinder}
\end{figure}
\end{proof}

The degenerate cases appear when $b$, the coefficient of $y$, is zero.  One can prove a high-dimensional version of \cref{thm:cylinder} by wrapping $\rr^d$ around $S^{d-1}\times \rr$ to find ``sinusoidal surfaces of fixed period'' that halve $d+1$ measures in $\rr^d$.

\section{Equipartitions with wedges and parallel hyperplanes}\label{sec:polyhedral-lift}

In this section, we prove results regarding equipartitions of $d+1$ mass distributions in $\rr^d$ by a wedge. 
A \textit{wedge} in $\R^d$ is the intersection of two closed half-spaces.  Note that a single closed half-space is also considered a wedge.

We say that a measure $\mu$ with support $K$ in $\rr^d$ is \textit{absolutely continuous} if it is absolutely continuous with respect to the Lebesgue measure, the interior of $K$ is connected and not empty, and for every open set $U \subset K$ we have $\mu(U) > 0$.  This guarantees that there is a unique halving hyperplane in each direction for $\mu$ and that the halving hyperplane varies continuously as we change the direction of the cut.  We first establish a lemma about halving hyperplanes.  We only use the lemma below with $n=d+1$, but it works in general.

\begin{lemma}\label{lem:separating-fixed-direction}
Let $n$ be a positive integer, $\mu_1, \ldots, \mu_{n}$ be finite absolutely continuous measures in $\rr^d$, and $v$ be a unit vector in $\rr^d$.  There either exists a hyperplane $H$ orthogonal to $v$ that halves each of the $n$ measures or there exists a hyperplane $H$ such that its two closed half-spaces satisfy
\begin{align*}
    \mu_i (H^+) & < \frac{1}{2}\mu_i(\rr^d) \qquad \mbox{for some $i \in [n]$ and}\\
    \mu_{i'} (H^-) & < \frac{1}{2}\mu_{i'}(\rr^d) \qquad \mbox{for some $i' \in [n]$.}
\end{align*}
\end{lemma}

\begin{proof}
For each $i$, let $H_i$ be the halving hyperplane for $\mu_i$ orthogonal to $v$.  If all these hyperplanes coincide we are done.  Otherwise, we can order this set of hyperplanes by the direction $v$ and any hyperplane $H$ strictly between the first $H_i$ and the last $H_{i'}$ satisfies the conditions we want.
\end{proof}

In the situation above, we always take the hyperplane $H'$ exactly half-way between the first $H_i$ and the last $H_{i'}$.  This makes the choice of hyperplanee continuous as $v$ varies, and invariant if we replace $v$ by $-v$.

Through the rest of the manuscript we denote the canonical basis of $\rr^d$ by $e_1, \ldots, e_d$.

\begin{theorem}\label{thm:simple-wedge}
Let $d$ be a positive integer and $\mu_1, \ldots, \mu_{d+1}$ be $d+1$ finite absolutely continuous measures in $\rr^d$.  Then, there exists a wedge that contains exactly half of each measure.
\end{theorem}

\begin{proof}
Let $v$ be a unit vector in $\rr^d$.  If there is a hyperplane $H$ orthogonal to $v$ halving each measure we are done.  Otherwise, by \cref{lem:separating-fixed-direction}, we can find a hyperplane $H$ such that each side contains less than half of some measure.

Consider the lifting of $\rr^d$ to $\rr^{d+1}$ where we append an additional coordinate to every point $x \in \rr^d$.  Formally, we lift via the map $x  \mapsto (x, \operatorname{dist}(x,H))$.

We denote by $S(H)$ the image of $\rr^{d}$ in this embedding.  Note that the function $x \mapsto \operatorname{dist}(x,H)$ is affine on each side of $H$, so $S(H)$ is contained in the union of two hyperplanes that contain $\{(x,0): x \in H\}$.  See \cref{fig:S(H)} for an illustration of the case $d=2$.

We lift each measure $\mu_i$ in $\rr^d$ to a measure $\sigma_i$ in $\rr^{d+1}$.  The measures $\sigma_1, \ldots, \sigma_{d+1}$ are no longer absolutely continuous.  We now apply the ham sandwich theorem for general measures in $\rr^{d+1}$.  Therefore, we can find a hyperplane $H'$ in $\rr^{d+1}$ so that its two closed half-spaces $(H')^+, (H')^-$ satisfy $\sigma_i((H')^+) \ge \frac{1}{2}\sigma_i(\rr^{d+1})$ and $\sigma_i((H')^-) \ge \frac{1}{2}\sigma_i(\rr^{d+1})$ for all $i =1,\ldots, d+1$.

By construction, each side of $H$ has strictly less than half of one of the measures $\mu_i$.  If the hyperplane $H'$ coincides with one of the two hyperplanes whose union contains $S(H)$, the half-space bounded by $H'$ that contains infinite rays in the direction $-e_{d+1}$ would have less than half the corresponding $\sigma_i$.  Therefore $H'$ is not one of the two hyperplanes forming $S(H)$.

As the two components of $S(H)$ were the only hyperplanes with non-zero measure for each $\sigma_i$, we conclude that $H'$ halves each of the measures in $\rr^{d+1}$.  As a final observation, $H'\cap S(H)$ projects back to $\rr^d$ as the boundary of a wedge that halves all measures.
\end{proof}

\begin{figure}
\centering
  \includegraphics[width=.8\textwidth]{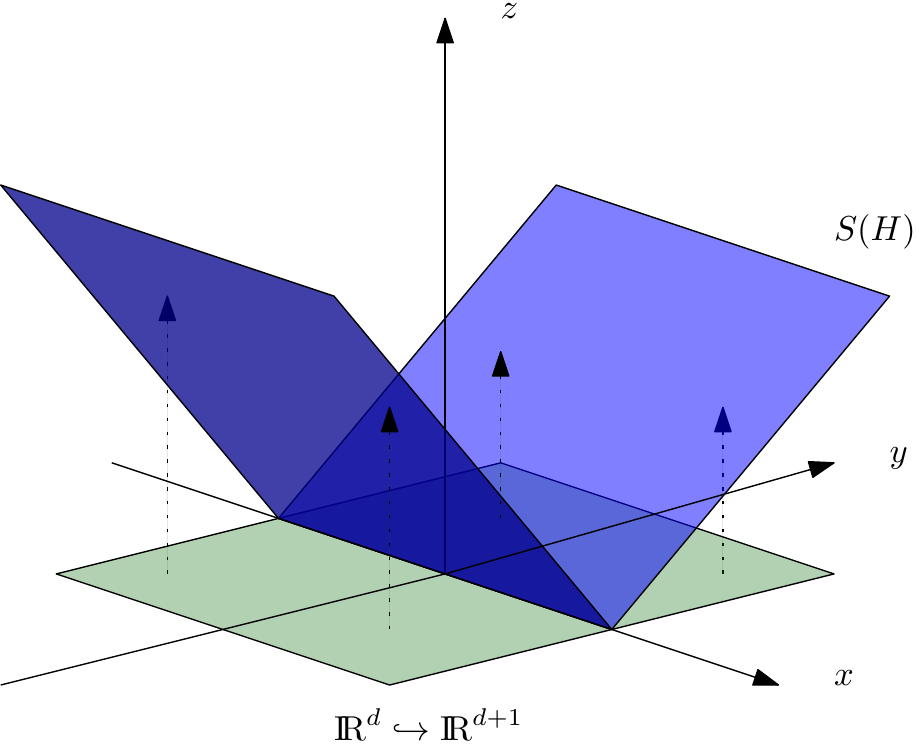}
  \caption{An example of lifting when given three sets of points in $\R^2$. Points on the $xy$-plane are sent to points on this surface. The hyperplane $H$ in this case is the $x$-axis.}
  \label{fig:S(H)}
\end{figure}
    

In the proof of \cref{thm:simple-wedge} we could choose the direction $v$ arbitrarily.  We now use this degree of freedom to strengthen the result.  Even though \cref{thm:parallel-hyperplanes} implies \cref{thm:simple-wedge}, we state it separately as the proof requires more technical tools.  In particular, a simple application of the ham sandwich theorem is insufficient.  We require some additional topological tools in lieu of the Borsuk--Ulam theorem.

Let $V_k(\rr^d)$ be the Stiefel manifold of orthonormal $k$-frames in $\rr^d$.  Formally,
\[
V_k(\rr^d) = \{(v_1,\ldots, v_k) : v_1, \ldots, v_k \in \rr^d \mbox{ are orthonormal}\}.
\]

The space $V_k(\rr^d)$ has a free action of the group $(\zz_2)^{k}$, where we consider $\zz_2 = \{+1,-1\}$ with multiplication.  Given $(v_1,\ldots, v_k) \in V_k(\rr^d)$ and $(\lambda_1, \ldots, \lambda_k) \in (\zz_2)^k$, we define 
\[
(\lambda_1, \ldots, \lambda_k)\cdot (v_1,\ldots, v_k) = (\lambda_1 v_1, \ldots, \lambda_k v_k) \in V_k(\rr^d).
\]

A similar action of $(\zz_2)^k$ can be defined in $\rr^{d-1}\times \ldots \times \rr^{d-k}$, as the direct product of the actions of $\zz_2$ on each $\rr^{d-i}$.  A recent result of Chan, Chen, Frick, and Hull describes properties of $(\zz_2)^k$-equivariant maps between these two spaces.

\begin{theorem}[Chan, Chen, Frick, Hull 2020 \cite{Chan2020}]\label{thm:topo1}
Let $k, d$ be positive integers.  Every continuous $(\zz_2)^k$-equivariant map $f:V_k(\rr^d) \to \rr^{d-1}\times \ldots \times \rr^{d-k}$ has a zero.
\end{theorem}

Manta and Sober\'on recently found an elementary proof of \cref{thm:topo1} \cite{Manta2021}.  We use the result above in \cref{sec:well separated}.  For this section, we need a slight modification.  We use the product of the actions of $\zz_2$ on the $d$-dimensional sphere $S^d$ and of $(\zz_2)^k$ on $V_k(\rr^d)$ to define a free action of $(\zz_2)^{k+1}$ on $S^d \times V_k(\rr^d)$.

\begin{theorem}\label{thm:new-topological-result}
Let $k, d$ be positive integers.  Every continuous $(\zz_2)^{k+1}$-equivariant map $f:S^d \times V_k(\rr^d) \to \rr^d \times \rr^{d-1}\times \ldots \times \rr^{d-k}$ has a zero.
\end{theorem}

There are several ways to prove the result above.  The dimension of the image and the domain are the same and the action of $(\zz_2)^{k+1}$ is free on $S^d \times V_k(\rr^d)$.  Therefore, \cref{thm:new-topological-result} is a consequence of the general Borsuk--Ulam type results of Musin \cite{Mus12}.  We simply need to find an equivariant function between these two spacse that has an odd number of orbits of zeroes.  Such functions are known for $\zz_2$-equivariant $f_1:S^d \to \rr^d$ and for $(\zz_2)^{k}$-equivariant $f_2: V_k(\rr^d) \to \rr^{d-1}\times \ldots \times \rr^{d-k}$, so we can simply take $f_0= f_1 \times f_2$.  Alternatively, one can use the methods of Chan et al. \cite{Chan2020} to prove \cref{thm:new-topological-result}.  It suffices to note that $S^d \times V_k(\rr^d)$ is a space in which their topological invariants can be applied and the particular function $f_0$ is all that's needed to replace \cite{Chan2020}*{second proof of Lemma 3.2}.  We only use \cref{thm:new-topological-result} for $k=d-1$.  

We present a short proof using the existing computations of the Fadell--Husseini index of these spaces on $\zz_2$ cohomology \cite{Fadell:1988tm}.  Given spaces $X$ and $Y$ with actions of $(\zz_2)^{k+1}$, their indices $\ind^{(\zz_2)^{k+1}}(X)$, $\ind^{(\zz_2)^{k+1}}(Y)$ are ideals in the polynomial ring $\zz_2[t_0,t_1,\ldots,t_k]$.  Moreover, if there exists a continuous $(\zz_2)^{k+1}$-equivariant map $f:X \to Y$, we must have $\ind^{(\zz_2)^{k+1}}(Y) \subset \ind^{(\zz_2)^{k+1}}(X)$.  More details on this index and its computation for spaces and group actions common in discrete geometry can be found on recent work of Blagojevi\'c, L\"uck, and Ziegler \cite{Blagojevic2015}.

\begin{proof}[Proof of \cref{thm:new-topological-result}]
	The result is equivalent to showing that there exists no continuous $(\zz_2)^{k+1}$-equivariant map $f:S^d \times V_k(\rr^d) \to (\rr^d \times \rr^{d-1}\times \ldots \times \rr^{d-k}) \setminus \{0\}$.  The space $(\rr^{d}\times \ldots \times \rr^{d-k}) \setminus \{0\}$ is homotopy equivalent to the join of spheres $S^{d-1}*\ldots * S^{d-k-1}$ and we know $\ind^{(\zz_2)^{k+1}} ( S^{d-1}*\ldots * S^{d-k-1} ) \subset \zz_2 [t_0, t_1,\ldots, t_k]$ is the ideal generated by the single monomial $t_0^dt_1^{d-1}\ldots t_{k}^{d-k}$.
	
	On the other hand, $\ind^{(\zz_2)^{k+1}} (S^d \times V_k(\rr^d))\subset \zz_2[t_0,\ldots, t_k]$ is the ideal generated by the polynomials $t_0^{d+1}, f_1,\ldots, f_k$ where $f_1,\ldots, f_k \subset \zz_2[t_1,\ldots, t_k]$ generate $\ind^{(\zz_2)^k}(V_k(\rr^d))$.  These polynomials are were described completely by Fadell and Husseini \cite{Fadell:1988tm}*{Thm. 3.16}.  Notably
	\[
	f_i = t_i^{n-i+1}+ w_{i,n-i}t_i^{n-i}+ \ldots + w_{i,0},
	\]
	where $w_{i,j} \in \zz_2[t_1,\ldots, t_{i-1}]$ and the degree of $w_{i,j}t_i^j$ is $n-i+1$.  In particular, 
	\[
	t_0^dt_1^{d-1}\ldots t_{k}^{d-k} \not\in \ind^{(\zz_2)^{k+1}}(S^d \times V_k(\rr^d)),
	\]
	which shows that no continuous $(\zz_2)^{k+1}$-equivariant map $f:S^d \times V_k(\rr^d) \to (\rr^d \times \rr^{d-1}\times \ldots \times \rr^{d-k})\setminus\{0\}$ exists.
\end{proof}

The second tool we require is a minor modification of the lift from $\rr^d$ to $\rr^{d+1}$.  In the previous proof, given a hyperplane $H \subset \rr^d$ we lifted $\rr^d$ directly to $S(H) \subset \rr^{d+1}$.  This has the inconvenience that the lift of an absolutely continuous measure is no longer absolutely continuous in $\rr^{d+1}$.  We do not require such a strong condition, but we do require the lifted measures to assign mass zero to any hyperplane.  To avoid this problem, we lift each measure $\mu_i$ to $S(H)^{\varepsilon}$, the region between $S(H)-\varepsilon \cdot e_{d+1}$ and $S(H)+\varepsilon \cdot e_{d+1}$, which we formalize below.

We say that a measure $\mu$ in $\rr^d$ is smooth if it is the integral of a continuous positive function $f:\rr^d \to \rr$ (i.e., $f(A) = \int_A f$ for any measurable set $A$).  We ``lift'' $f$ to a function 
\begin{align*}
    \tilde{f} : \rr^d \times \rr & \to \rr \\
    (x,t) & \mapsto \begin{cases}
    \left(\frac{1}{2\varepsilon}\right)f(x) & \mbox{ if } |\operatorname{dist}(x,H) - t| \le \varepsilon \\
    0 & \mbox{ otherwise.}
    \end{cases}
\end{align*}

We say that the measure $\sigma^{\varepsilon}$ defined as the integral of $\tilde{f}$ in $\rr^d$ is the lift of $\mu$ to $S(H)^{\varepsilon}$.  Notice that as $\varepsilon \to 0$, the measure $\sigma^{\varepsilon}$ converges weakly to the lift of $\mu$ to $S(H)$.  For $\varepsilon>0$, the measure $\sigma^{\varepsilon}$ is not absolutely continuous, but it has value $0$ on each hyperplane.

\begin{proof}[Proof of \cref{thm:parallel-hyperplanes}]

We first assume that no hyperplane simultaneously halves all measures, or we are done.  Since the set of smooth measures is dense in the set of absolutely continuous measures, we may assume without loss of generality that the measures $\mu_1, \ldots, \mu_{d+1}$ are smooth.  Let $\varepsilon>0$.  For $v \in S^d$ and $(v_1, \ldots, v_{d-1}) \in V_{d-1}(\rr^d)$, consider the element $(v,v_1, \ldots, v_{d-1}) \in S^d \times V_{d-1} (\rr^d)$.

Let $H$ be the translate of the hyperplane $T=\operatorname{span}\{v_1,\ldots, v_{d-1}\}$ chosen from \cref{lem:separating-fixed-direction}.  Let $\sigma^{\varepsilon}_1, \ldots, \sigma^{\varepsilon}_{d+1}$ be the lifts of $\mu_1, \ldots, \mu_{d+1}$ to $S(H)^{\varepsilon}$, respectively.

Let $\lambda$ be the value so that the half-spaces 
\begin{align*}
    A & = \{x \in \rr^{d+1}: \langle x, v \rangle \ge \lambda \} \quad \mbox{ and } \\
    B & = \{x \in \rr^{d+1}: \langle x, v \rangle \le \lambda \}
\end{align*}

have the same $\sigma_{d+1}^{\varepsilon}$-measure.  Now we are ready to define a map
\begin{align*}
    f : S^d \times V_{d-1}(\rr^d) & \to \rr^d \times \rr^{d-1} \times \ldots \times \rr^1 \\
(v,v_1,\ldots, v_{d-1}) & \mapsto (x_d,\ldots, x_1)
\end{align*}
where $x_i \in \rr^{i}$ for $i=1,\ldots, d$.  First, we consider
\[
x_d = \begin{bmatrix}
\sigma_1^{\varepsilon}(A)-\sigma_1^{\varepsilon}(B) \\
\vdots \\
\sigma_d^{\varepsilon}(A)-\sigma_d^{\varepsilon}(B)
\end{bmatrix}.
\]
For $i=1,\ldots, d-1$, the first coordinate of $x_i$ is $\langle v, e_{d+1}\rangle \langle v, (v_{d-i},0) \rangle$ and the rest are zero.

This function is continuous by the construction of the lift of the measures.  It is also $(\zz_2)^{k+1}$-equivariant: if we flip the sign of $v$, only $x_d$ changes sign and if we flip the sign of $v_i$ only $x_{d-i}$ changes sign for $i=1,\ldots, d-1$.

By \cref{thm:new-topological-result} the function $f$ has a zero.  The condition $x_d = 0$ tells us that $A$ and $B$ each have exactly half of each $\sigma^{\varepsilon}_i$ for $i=1,\ldots, d+1$.  The conditions of the rest of the $x_i$ being zero vectors means that either $v$ is orthogonal to $e_{d+1}$ or $v$ is orthogonal to each $(v_i,0)$ for $i=1,\ldots, d-1$.

If the first condition happens, then when we project $\rr^{d+1}$ to the hyperplane $e_{d+1} = 0$, each $\sigma^{\varepsilon}_i$ projects to $\mu_i$ and $A$, $B$ project onto two half-spaces of $\rr^d$.  This would mean we have a hyperplane halving each of the original measures, contradicting our initial assumption.  Therefore $v$ is orthogonal to $(v_i,0)$ for all $i$.

We take a sequence of positive real numbers $\varepsilon_{k} \to 0$.  For each of them, we find a zero of the function induced above.  As $S^d \times V_{d-1}(\rr^d)$ is compact, the zeros must have a converging subsequence.  In the limit, we obtain two complementary half-spaces $A, B$ so that each contains at least half of $\sigma_i$ for $i=1,\ldots, d+1$ on $S(H)$ (where the direction of $H$ is determined by the limit of $(v_1, \ldots, v_{d-1})$).  Let $H'$ be the hyperplane at the boundary of $A$ and $B$.

In the limit, the vector $v$ normal to $H'$ is orthogonal to the subspace $T_0=\operatorname{span}\{(v_1,0), \ldots, (v_{d-1},0)\}$.  By the construction of $H$, the hyperplane $H'$ cannot be one of the two hyperplane components of $S(H)$.  The orthogonality mentioned before implies that $H' \cap S(H)$ must be two $(d-1)$-dimensional affine spaces parallel to $T_0$.  When we project back to $\rr^d$, these parallel intersections form $H_1$ and $H_2$ and the region between them has exactly half of each $\mu_i$.
\end{proof}

\begin{proof}[Proof of \cref{cor:Bagel}]
We lift $\rr^d$ to the paraboloid 
\[
\mathcal{P} = \left\{\left(x_1,\ldots, x_d, \sum_{i=1}^d x_i^2\right)\in \rr^{d+1}\right\}.
\]
We now have $d+2$ measures in $\rr^{d+1}$, and every hyperplane has measure zero in each of them.  We can apply \cref{thm:parallel-hyperplanes} and find two parallel hyperplanes $H_1, H_2$ so that the region between them contains exactly half of each measure.  Notice that $H_1 \cap \mathcal{P}$ and $H_2 \cap \mathcal{P}$ project onto concentric spheres in $\rr^d$.
\end{proof}

Of course, similar results can be obtained by applying general Veronese maps as Stone and Tukey did to prove the polynomial ham sandwich theorem \cite{Stone:1942hu}.

\begin{corollary}
Let $d,k$ be positive integers.  For any set of $\binom{d+k}{k}$ finite absolutely continuous measures in $\rr^d$, there exists a polynomial $P$ on $d$ variables and constants $\lambda_1, \lambda_2$ so that $\{x \in \rr^d : \lambda_1 \le P(x) \le \lambda_2\}$ has exactly half of each measure.
\end{corollary}

The polynomial ham sandwich usually splits $\binom{d+k}{k}-1$ measures using a polynomial of degree at most $k$.  If we want to restrict which monomials are used in the splitting polynomial, we just have to reduce the number of measures accordingly.

\section{Fixed size partitions for well separated measures}\label{sec:well separated}

As noted by B\'ar\'any et al. \cite{Barany:2008vv}, it is known that for well separated convex subsets $K_1, \ldots, K_d$  in $\rr^d$, there are $2^d$ hyperplanes tangent to all of them.  If none of the hyperplanes are vertical (i.e., perpendicular to $e_d$), the tangent hyperplanes are in one-to-one correspondence with the sets of $K_i$ below the hyperplane.  This is was also proved by Klee, Lewis, and Hohenbalken \cite{Klee1997}.  We use this fact in our proof of \cref{thm:BHJ}.

\begin{proof}[Proof of \cref{thm:BHJ}]
Consider the hypercube $Q = [0,1]^d$.  Each vertex of $Q$ can be assigned to a subset of $I \subset [d]$ uniquely.  We denote
\[
v_I = (p_1,\ldots, p_d) \qquad \mbox{where } \quad p_i = \begin{cases}
	1 & \mbox{ if } i\in I \\
	0 & \mbox{ if } i\not\in I.
\end{cases}
\]

For a point $q = (q_1, \ldots, q_d) \in Q$ and $I \subset [d]$ we consider the coefficients
\begin{align*}
\lambda_I(q) = \prod_{i \in I}q_i \prod_{i \not\in I} (1-q_i).
\end{align*}

The coefficients $\lambda_I(q)$ are the coefficients of a convex combination, as they are non-negative and their sum is $1$.  Suppose we have a function $f: \{0,1\}^d \to \rr^d$.  We can extend it to a function $\tilde{f}:Q \to \rr^d$ by mapping
\[
q \mapsto \sum_{I \subset [d]} \lambda_I(q) f(v_I).
\]

Notice that if $\sigma$ is a face of $Q$, then $\tilde{f}(\sigma) \subset \conv(\{f(v_I): v_I \in \sigma\})$.  In particular, $\tilde{f}(v_I) = f(v_I)$.

Now, suppose we are given $d$ well separated convex sets $K_1, \ldots, K_d$ in $\rr^d$ and measures $\mu_1, \ldots, \mu_d$ so that the support of $\mu_i$ is $K_i$.  We may assume without loss of generality that there is no vertical hyperplane tangent to each of $K_1, \ldots, K_d$.

Each non-vertical hyperplane $H$ can be written as
\[
\{(x_1, \ldots, x_d) : x_d = \alpha_1 x_1 + \ldots + \alpha_{d-1} x_{d-1} + \alpha_d \}
\]
for some constants $\alpha_1, \ldots, \alpha_d$.  We assign the vector $r(H)=(\alpha_1, \ldots, \alpha_d)$ to the hyperplane $H$.  We say that a point is above $H$ if $x_d \ge \alpha_1 x_1 + \ldots + \alpha_{d-1} x_{d-1} + \alpha_d$ and below $H$ if $x_d \le \alpha_1 x_1 + \ldots + \alpha_{d-1} x_{d-1} + \alpha_d$.  Notice that if a point $x$ is below a set of hyperplanes $H_1, \ldots, H_k$, then it is also below the hyperplane $r^{-1}(y)$ for any $y \in \conv (r(H_1), \ldots, r(H_k))$.

We know that for each subset $I \subset [d]$, there is a unique hyperplane $H_I$ that is tangent to each $K_i$ and so that $K_i$ is below $H$ if and only if $i \in I$.  This defines a function $f:\{0,1\}^d \to \rr^d$ by simply taking $f(v_I) = r(H_I)$.  We extend $f$ to a function $\tilde{f}:Q \to \rr^d$ as described above.  For $q \in Q$, let $H(q)$ be the set of points below $r^{-1}(\tilde{f}(q))$.

We define a final function
\begin{align*}
	g: Q & \to Q \\
	q & \mapsto (\mu_1 (H(q)), \ldots, \mu_d (H(q)))
\end{align*}

The function $g$ is continuous.  Notice, for example, that $g(v_I) = v_I$ for each $I \subset [d]$.  Moreover, using the properties of $\tilde{f}$, we have that for every face $\sigma \subset Q$, we have $g(\sigma) \subset \sigma$.  This means that $g$ is of degree one on the boundary, so it is surjective.  In particular, there is a point $q_0 \in Q$ such that $g(q_0) = (\alpha_1, \ldots, \alpha_d)$.  Therefore, the hyperplane $H(q_0)$ is the hyperplane we were looking for.
\end{proof}

B\'ar\'any, Hubard, and Karasev also showed that under simple conditions the half-space $H$ from \cref{thm:BHJ} is unique. It suffices that
\begin{itemize}
    \item each measure $\mu_i$ assigns a positive value to each open set in its support $K_i$,
    \item the interior of each $K_i$ is connected and not empty,
    \item no vertical hyperplane is tangent to $K_1, \ldots, K_d$, and
    \item the half-space $H$ contains infinite rays in direction $-e_d$.
\end{itemize}

\begin{proof}[Proof of \cref{thm:parallel-hyperplanes-separated}]
We follow a process similar to the proof of \cref{thm:parallel-hyperplanes}.  First we need an additional observation about our construction of $S(H)$.  In $\rr^d$, there is no hyperplane $H$ intersecting each of $K_1, \ldots, K_{d+1}$.  Otherwise, we can take a point $p_i \in K_i \cap H$.  This gives us $d+1$ points in a $(d-1)$-dimensional space, so by Radon's lemma we can find a partition of them into two subsets $A, B$ whose convex hulls intersect.  This implies that $\{K_i : p_i \in A\}$ cannot be separated from $\{K_i : p_i \in B\}$, contradicting the hypothesis.

Therefore, when we are given an vector $v \in \rr^d$ and we construct $S(H)$ for $H \perp v$, each side of $H$ must have measure zero for some $\mu_i$.  This is much stronger than simply having less than half of some measure.

The main idea will be to lift each measure to a surface $S(H)$ for an appropriate $H$ and use \cref{thm:BHJ}.  We show that by choosing $H$ carefully, we can deduce the existence of the two parallel hyperplanes we seek.

Consider the Stiefel manifold $V_{d-1}(\rr^d)$.  Given $(v_1, \ldots, v_{d-1}) \in V_{d-1}(\rr^d)$, we lift $\rr^d$ to $S(H) \subset \rr^{d+1}$ as in \cref{lem:separating-fixed-direction} where $H$ is parallel to $\operatorname{span}\{v_1,\ldots, v_{d-1}\}$.  This lifts each measure $\mu_i$ in $\rr^d$ to a measure $\sigma_i$ in $S(H)$.  Every hyperplane in $\rr^d$ separating the support vectors of two sets of the measures $\mu_{i}$ can be extended vertically in $\rr^{d+1}$ to separate the corresponding measures $\sigma_i$.  A small tilting can ensure that the separating hyperplane is not vertical.  Therefore, the measures $\sigma_i$ are well separated.

The measures $\sigma_i$ do not satisfy the requirements of \cref{thm:BHJ}, so an additional step is necessary.  For an $\varepsilon>0$, we lift the measures to $S(H)^{\varepsilon}$ as in the proof of \cref{thm:parallel-hyperplanes}, apply \cref{thm:BHJ}, and then take $\varepsilon \to 0$.

This ensures that we get a half-space $H^+$ that has infinite rays in the direction $-e_{d+1}$ and such that $\sigma_i(H^+) \ge \alpha_i \cdot \sigma_i(\rr^d)$ and $\sigma_i(H^-) \ge (1-\alpha_i) \cdot \sigma_i(\rr^d)$, where $H^-$ is the complementary closed half-space of $H^+$.  Since $\sigma_i(H^+)>0$ for all $i$, we know that the boundary of $H^+$ cannot be one of the two hyperplane components of $S(H)$.  Therefore the measure $\sigma_i$ of the boundary of $H^+$ is zero for all $i$ and $\sigma_i(H^+) = \alpha_i \cdot \mu_i(H^+)$.  The same arguments that B\'ar\'any, Hubard, and Jer\'onimo used to show the uniqueness in their theorem can be applied to show that $H^+$ is uniquely defined.  The uniqueness also implies that $H^+$ changes continuously as we modify $(v_1, \ldots, v_{d-1})$.

Let $n \in S^d \subset \rr^{d+1}$ be the normal vector to the boundary of $H^+$ that points in the direction of $H^+$.  We can use this to construct a function
\begin{align*}
    g: V_{d-1} (\rr^{d}) & \to \rr^{d-1} \times \ldots \times \rr^{1} \\
    (v_1, \ldots, v_{d-1}) & \mapsto (x_1, \ldots, x_d).
\end{align*}
For each $i=1,\ldots, d-1$, the first coordinate of $x_i \in \rr^{d-i}$ is $\langle n, (v_i,0) \rangle$ and the rest are zero.  This function is well defined and continuous.  If we flip the sign of $v_i$, the surface $S(H)$ does not change.  The vector $n \in S^d$ is not affected by this change, so only the sign of $x_i$ changes.  Therefore, the function $g$ is $(\zz_2)^{d-1}$-equivariant.  By \cref{thm:topo1}, the function $g$  must have a zero.  This implies that the projection of $H^+ \cap S(H)$ onto $\rr^d$ is the region between two hyperplanes parallel to $H$.
\end{proof}

The construction of the function $g$ only uses $d-1$ out of the $d(d-1)/2$ coordinates that \cref{thm:topo1} makes available.  It would be interesting to know if much stronger conditions can be imposed on $H$.

We also have consequences similar to \cref{cor:Bagel}.  We say that a family of sets $K_1, \ldots, K_{d+2}$ in $\rr^d$ is \textit{well separated by spheres} if for any subset way to split them into two families $I, J$, there is a sphere that separates $I$ and $J$, i.e., it contains the union of one of the sets and leaves out the union of the other set.

\begin{corollary}
Let $d$ be a positive integer and $\mu_1, \ldots, \mu_{d+2}$ be measures in $\rr^d$ absolutely continuous with respect to the Lebesgue measure.  Suppose that the supports $K_1, \ldots, K_{d+2}$ of $\mu_1, \ldots, \mu_{d+2}$ are well separated by spheres.
Let $\alpha_1, \ldots, \alpha_{d+2}$ be real numbers in $(0,1)$. Then, there exist two concentric $S_1, S_2$ in $\rr^d$ so that the region $A$ between them satisfies
\[
\mu_i(A) = \alpha_i \cdot \mu_i(\rr^d) \qquad \mbox{for all }i=1,\ldots, d+2.
\]
\end{corollary}

\begin{proof}
We lift $\rr^d$ to the paraboloid
\[
\mathcal{P} = \left\{\left(x_1,\ldots, x_d, \sum_{i=1}^d x_i^2\right)\in \rr^{d+1}\right\}.
\]
A sphere in $\rr^d$ separating two families $I, J$ of measure supports translates to a hyperplane in $\rr^{d+1}$ separating the lift of those supports.  We apply \cref{thm:parallel-hyperplanes-separated} to the family of measures induced on $\mathcal{P}$ and we are done.

Even though the set of lifted measures do not satisfy the conditions of \cref{thm:parallel-hyperplanes-separated}, a standard approximation argument fixes this problem.
\end{proof}

\section{Equipartition with Polytopes and polyhedral surfaces of bounded complexity}\label{sec:polyhedral}
 In previous sections, the number of measures to be partitioned was constrained by the dimension of the ambient space, while the boundaries of the partition were relatively simple. In this section we consider mass partitions of a family of $n$ measures in $\R^d$, where $n$ can be much larger than $d$.  We do so by increasing the complexity of the boundary of the partition.  We focus on partitions by polyhedral surfaces.

\begin{definition}\label{def:nicely-separated}
Let $\mathcal{F}=\{\mu_1,\ldots, \mu_n\}$ be a family of finite absolutely continuous measures in $\R^d$ with support $K_i$ for each $1\leq i\leq n$. The supports are called \textit{nicely separated} if for each $1\leq i\leq n$, there exists a hyperplane $H_i$ such that $K_i\cap H_i^+=\emptyset$ and $K_j\cap H_i^-=\emptyset$ for all $j\neq i$.
\end{definition}
The maximum number of well separated measures is $d+1$, due to Radon's theorem.  For nicely separated measures we only want to be able to separate any measure for the union of the other $n-1$, and not any two subsets.  An example of nicely separated measures are $n$ measures such that each is concentrated near a vertex of a convex polytope.
\begin{figure}[h!]
\centering
  \includegraphics[width=1.2\textwidth]{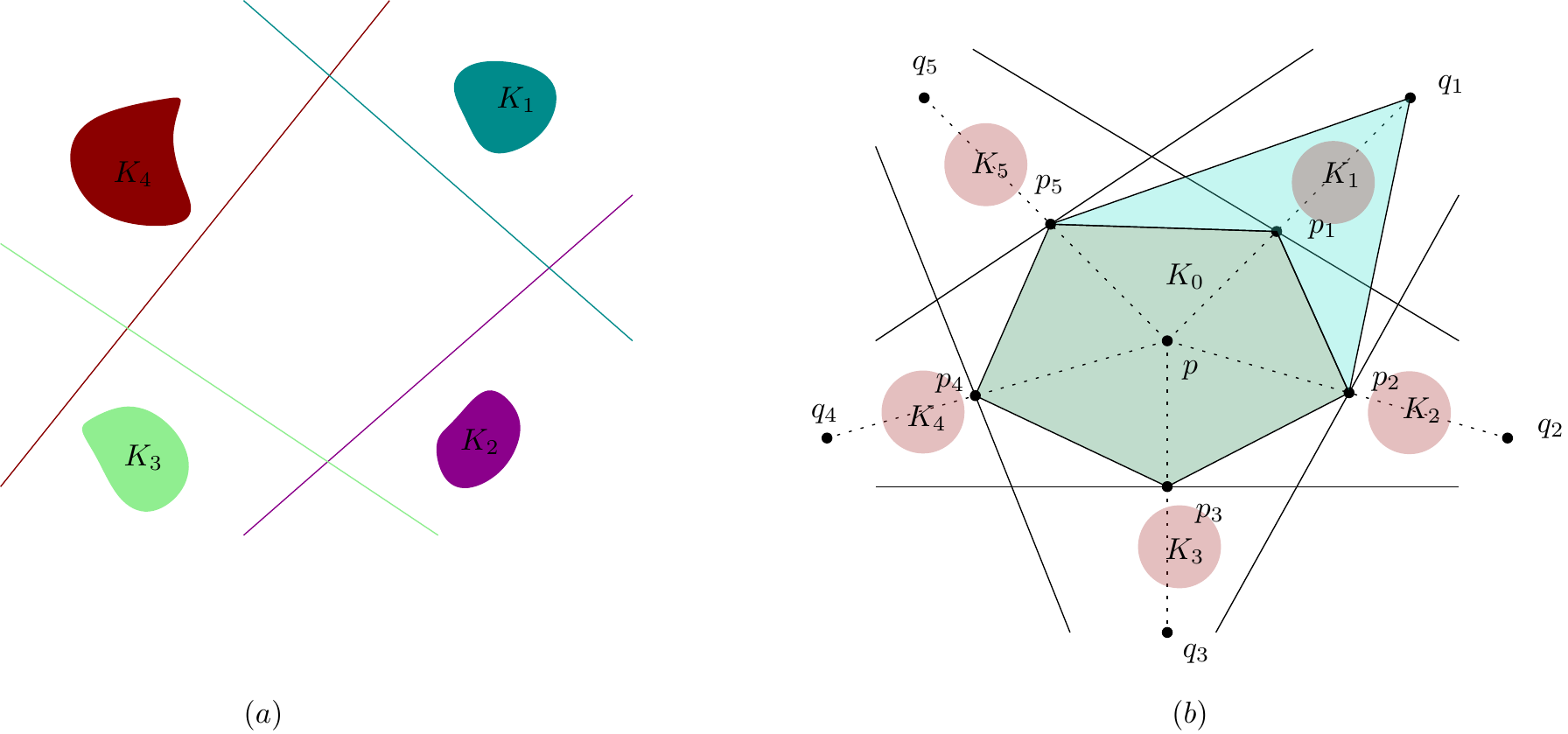}
  \caption{(a) An example of four nicely separated measures in $\rr^2$.  (b) An example of five nicely separated and concentrated measures in $\rr^2$.  Notice that if we take $q_1$ instead of $p_1$ to form the convex hull, the resulting polygon contains all of $K_1$.}
  \label{fig:nicely-separated}
\end{figure}

 We define a polyhedron in $\rr^d$ to be a finite intersection of closed half-spaces. A facet of a polyhedron is a $(d-1)$-dimensional face, and a vertex of a polyhedron is a zero dimensional face.

\begin{theorem}\label{thm:n-faces}
Let $\mathcal{F}=\{\mu_1,\ldots, \mu_n\}$ be a family of finite absolutely continuous measures in $\R^d$ with nicely separated supports $K_i$ for all $1\leq i\leq n$, and let $\alpha_1,\ldots,\alpha_{n}$ be real numbers in $(0,1)$. Then, there exists a polyhedron $P$ with at most $n$ facets such that $\mu_i(P)=\alpha_i\cdot \mu_i(\R^d)$ for every $1\leq i\leq n$.
\end{theorem}
\begin{proof}
Because the supports are nicely separated, for each $1\leq i \leq n$, we can fix a hyperplane $H_i$ with $K_i\cap H_i^+ = \emptyset$ and $K_j\cap H_i^-=\emptyset$ for all $j\neq i$. Notice that a polyhedron $P = \bigcap_{ i=1}^n H_i^+$ has the property $\mu_i(P) = 0$ for every $1\leq i \leq n$. 

Now, consider $\mu_1$. Let $v$ be the normal vector to the hyperplane $H_1$ pointing in the direction of $H^-$. We can move $H_1$ in the direction of $v$ until we have the desired portion of the measure $\mu_1$, so we can fix $H_1'\parallel H_1$ with $\mu_1(H_1'^+) = \alpha \cdot \mu_1(\R^d)$. By letting $P' = \left(\bigcap_{i=2}^n H_i^+\right) \cap H_1'^+ $, we have $\mu_1(P') = \alpha_1 \cdot \mu_1(\R^d)$ because $\mu_1\big(\bigcap_{i=2}^{n}H_i^+\big) = \mu_1(\R^d)$.
Moreover, because $K_j\cap H_1^-=\emptyset$ for each $j\neq 1$, moving $H_1$ to the direction of $H_1^-$ does not interfere with the rest of the measures $\mu_2,\ldots,\mu_n$. We can repeat the same process for $\mu_2,\ldots,\mu_n$ to find a convex polyhedron of at most $n$ facets with the desired property.
\end{proof}

While \cref{thm:n-faces} allows for a mass partition with a polyhedron of $n$ facets, we can quantify the complexity of a compact polyhedron by the number of vertices as well. \cref{thm:n-vertices} proves a mass partition with a polyhedron of $n$ vertices, but this time for $n$ measures with a slightly stronger separation condition. We will use a similar idea to the proof of \cref{thm:BHJ}.

Let $\mu_1, \ldots, \mu_n$ be a family of nicely separated measures in $\rr^d$.  Let $H_i$ be the hyperplane separating $K_i$ from the rest of the supports, as in \cref{def:nicely-separated}.  For $n \ge d+1$, by Helly's theorem we know that $P=\bigcap_{i=1}^n H^+_i \neq \emptyset$. We say that the measures are \textit{concentrated} if the following happens.  There exists a point $p \in P$ and points $p_i, q_i$ for $i = 1,\ldots, n$ so that the following holds.
\begin{itemize}
    \item For each $i=1,\ldots,n$, $p_i \in H_i \cap P$.  We denote $K_0=\conv\{p_1,\ldots, p_n\}$.
    \item We have $p \in K_0$.
    \item For each $i=1,\ldots, n$, $q_i$ is in the ray $pp_i$ and in $\bigcap_{i' \neq i}H^+_{i'}$.
    \item For each $i=1,\ldots, n$, we have $K_i \subset \conv (\{q_i\}\cup K_0)$.
\end{itemize}

An example is illustrated in \cref{fig:nicely-separated}(b).

\begin{theorem}\label{thm:n-vertices}
Let $n,d$ be positive integers.  Let $\mathcal{F}=\{\mu_1,\ldots, \mu_n\}$ be a family of nicely separated and concentrated measures in $\R^d$, each absolutely continuous.  Let $ \alpha_1,\ldots,\alpha_{n}$ be real numbers in $(0,1)$. Then, there exists a polytope $K$ with $n$ vertices such that $\mu_i(K)=\alpha_i\cdot \mu_i(\R^d)$ for every $1\leq i\leq n$.
\end{theorem}

Note that the intuitive idea we used to prove \cref{thm:n-faces} would indicate that we should slide each $p_i$ towards $q_i$ until we have the desired measure.  The issue with this is that the values of other measures are no longer fixed.

\begin{proof}
Consider the hypercube $Q=[0,1]^n$.  For $x=(x_1,\ldots, x_n) \in Q$, and $i=1,\ldots, n$, let $y_i = (1-x_i)q_i+x_i p_i$.  We define
\[
K(x) = \conv \{y_1,\ldots, y_n\}.
\]
This convex set allows us to construct a function
\begin{align*}
    f: Q & \to Q \\
    x & \mapsto \left( \frac{\mu_1(K(x))}{\mu_1(\rr^d)}, \ldots, \frac{\mu_n(K(x))}{\mu_n(\rr^d)}\right).
\end{align*}

The function is continuous.  From the conditions of the measures, we can see that for every vertex $v$ of $Q$, we have $f(v) = v$.  However, we have a stronger condition.  For every face $\sigma \subset Q$, we have $f(\sigma) \subset \sigma$.  This is because if a coordinate $x_i$ of $x$ equals zero, the $K(x) \subset H_i^+$, so $\mu_i(K(x)) = 0$.  If $x_i=1$, then $K(x) \supset \conv\{\{q_i\}\cup K_0\}$, so $\mu_i(K(x)) = \mu_i (\rr^d)$.  Therefore $f$ is of degree one on the boundary and must be surjective.  There is a point $x \in Q$ such that $f(x) = (\alpha_1,\ldots,\alpha_n)$, which implies that $K(x)$ is the polytope we were looking for.
\end{proof}

\section{remarks and open problems}\label{sec:remarks}

To prove \cref{thm:same-fraction}, we need to strengthen \cref{lem:separating-fixed-direction}.

\begin{lemma}\label{lem:separating-fixed-direction-strong}
Let $m, n$ be positive integers, $\mu_1, \ldots, \mu_{n}$ be $n$ finite absolutely continuous measures in $\rr^d$, and $v$ be a unit vector in $\rr^d$.  There either exists $m-1$ hyperplanes orthogonal to $v$ that divide $\rr^d$ into $m$ regions $R_1,\ldots, R_m$ of equal measure for each $\mu_i$ simultaneously or there exist $m-1$ hyperplanes orthogonal to $v$ such that they divide $\rr^d$ into $m$ regions $R_1, \ldots, R_m$ such that for every $j=1,\ldots, m$ there exists an $i$ such that
\begin{align*}
    \mu_i (R_j) & < \frac{1}{m}\mu_i(\rr^d).
\end{align*}
\end{lemma}

\begin{proof}
Given parallel hyperplanes $H_1, \ldots, H_{m-1}$ in this order, we denote by $R_1, \ldots, R_m$ the regions they divide $\rr^d$ into such that $R_j$ is bounded by $H_{j-1}$ and $H_j$.  The unbounded regions $R_1$, $R_m$ are bounded only by $H_1$ and $H_m$ respectively.

We can find $m-1$ hyperplanes such that $\mu_1(R_j) = (1/m)\mu_1(\rr^d)$ for every $j$.  If these regions also form an equipartition for every other $\mu_i$, we are done.  Otherwise, there is an $i$ and a $j$ such that $\mu_i(R_j) < (1/m)\mu_i(\rr^d)$.  We can widen $R_j$ by moving $R_{j-1}$ and $R_j$ slightly apart so that we still have  $\mu_i(R_j) < (1/m)\mu_i(\rr^d)$.

Then, $\mu_1(R_{j-1}) < (1/m) \mu_1(\rr^d)$ and $\mu_1(R_{j+1}) < (1/m) \mu_1(\rr^d)$.  We can translate $H_{j-2}$ and $H_{j+1}$ away from $H_{j-1}$ and $H_j$ respectively so that these inequalities are preserved.  This makes $\mu_1(R_{j-2})$ and $\mu_1(R_{j+2})$ to be strictly reduced.  We continue this way until we are done. 
\end{proof}

Now, given $\mu_1, \ldots, \mu_{d+1}$ finite absolutely continuous measures in $\rr^d$, we construct a surface in $\rr^{d+1}$.  We take $v=e_{d}$ and find the $m-1$ hyperplanes $H_1, \ldots, H_m$ such that
\[
H_j = \{(x_1,\ldots, x_d) \in \rr^d: x_d = \lambda_j\}.
\]
For some $\lambda_1 < \ldots < \lambda_{m-1}$.  We define $\lambda_0 =-\infty$ and $\lambda_m = \infty$.  Let $h:\rr \to \rr$ be a convex function that is linear between $\lambda_j$ and $\lambda_{j+1}$ for each $j=0,\ldots, {m-1}$, but not between $\lambda_j$ and $\lambda_{j+2}$ for each $j=0,\ldots, m-2$.

Let $V$ be the surface in $\rr^{d+1}$ defined by the equation $x_{d+1}=h(x_d)$.  The set of points on or above $V$ is the intersection of $m$ closed half-spaces.  To prove \cref{thm:same-fraction} we repeat the proof of Akopyan and Karasev but we lift $\rr^d$ to $V$ instead of a paraboloid.

\begin{proof}[Proof of \cref{thm:same-fraction}]
By a subdivision argument, it suffices to prove the result when $n=p$ a prime number.  We apply \cref{lem:separating-fixed-direction-strong} with $m=p$.  If there are $p-1$ parallel hyperplanes that form an equipartition of the measures, we are done.  Otherwise, we lift $\rr^d$ to $\rr^{d+1}$ by lifting it to the surface $V$ defined above.  Let $\sigma_1, \ldots, \sigma_{d+1}$ be the measures induced by $\mu_1,\ldots, \mu_{d+1}$ on $V$.  It's known that we can split $\rr^{d+1}$ into $p$ convex sets $C_1, \ldots, C_p$ that form an equipartition of $\mu_1, \ldots, \mu_{d+1}$.  Since each of the regions $R_1, \ldots, R_p$ we constructed in $\rr^{d}$ have less than a $(1/p)$-fraction of some $\mu_i$ and $V$ is the boundary of a convex set, none of the boundaries between the sets $C_{j}$ can coincide with the hyperplanes defining $V$.

Moreover, the sets are $C_1, \ldots, C_p$ are induced by a generalized Voronoi diagram \cites{Karasev:2014gi, Blagojevic:2014ey}.  In other words, there are points (called sites) $s_1, \ldots, s_p$ in $\rr^{d+1}$ and real number $\beta_1, \ldots, \beta_p$ such that the $p$ convex regions
\[
C_j = \{x \in \rr^{d+1}: ||x-s_j||^2-\beta_j \le ||x - s_{j'}||^2-\beta_{j'} \mbox{ for } j' =1,\ldots, p\}
\]
form an equipartition of $\mu_1, \ldots, \mu_{d+1}$.  Since the set of point above $V$ is convex, if we take the region $C_j$ whose site $s_j$ has minimal $(d+1)$-th coordinate, when we project $C_j \cap V$ back to $\rr^d$ we get a convex set.  This is the set $K$ we are looking for.  The boundary of the corresponding $C_j$ is the union of at most $p-1$ hyperplanes (the ones dividing it from each other $C_{j'}$).  Each of those $p-1$ hyperplanes can intersect each of the $p$ hyperplanes defining $V$, forming at most $p(p-1)$ linear components of the boundary of $C_j \cap V$.  This gives us the bound on the number of half-spaces whose intersection is $K$.
\end{proof}

Earlier proofs of the equipartition result we used do not guarantee that the partition $C_1, \ldots, C_p$ comes directly from a generalized Voronoi diagram \cite{Soberon:2012kp}, which is important in this proof.  When $n$ is a prime power, the number of half-spaces we used grows logarithmically with $n$.  We wonder if this holds in general.

\begin{question}
Let $d$ be a fixed integer.  Determine if for every positive integer $n$ and any $d+1$ finite absolutely continuous measures $\mu_1, \ldots, \mu_{d+1}$ in $\rr^d$ there exists a convex set $K \subset \rr^d$ formed by the intersection of $O(\log n)$ half-spaces that contains exactly a $(1/n)$-fraction of each $\mu_i$.
\end{question}

We nicknamed \cref{cor:Bagel} the bagel ham sandwich theorem due to its drawing in $\rr^2$.  However, since the set used is the region between two concentric spheres, it certainly does not look like a Bagel in $\rr^3$.  We define a \textit{regular torus} in $\rr^3$ to be the any set of the form $\{x \in \rr^3: \operatorname{dist}(x,S)\le \alpha\}$ where $S$ is a flat circle in $\rr^3$ and $\alpha$ is a positive real number.

\begin{question}[Three-dimensional bagels]
Is is true that for any five absolutely continuous finite measures in $\rr^3$ there exists a regular torus containing exactly half of each measure?
\end{question}

With four measures the result holds, since when $S$ degenerates to a point the regular torus is a sphere.

One of the questions that motivated the work on this manuscript was inspired by a conjecture by Mikio Kano.  Kano conjectured that for any $n$ smooth measures in $\rr^2$ there exists a path formed only by horizontal and vertical segments, that takes at most $n-1$ turns, that simultaneously halves each measure.  The conjecture is only known for $k=1,2$ or if the path is allowed to go through infinity \cites{Uno:2009wk, Karasev:2016cn}.  We wonder if the following way to mix Kano's conjecture with \cref{thm:circular-sandwich} holds.

\begin{question}[Existence of square sandwiches]\label{question:square}
Is is true that for any three finite absolutely continuous measures in $\rr^2$ there exists a square that contains exactly half of each measure?
\end{question}

\cref{thm:parallel-hyperplanes} shows that we have a positive answer for rectangles (if the support of the measures are compact, we can cut the two lines given by \cref{thm:parallel-hyperplanes} by perpendicular segments sufficiently far away, otherwise we have degenerate rectangles).  However, it is still possible that for squares the answer to \cref{question:square} is affirmative.


\end{document}